\documentclass[11pt]{article}

\usepackage{amsmath,amssymb,latexsym}
\usepackage{amsthm}
\usepackage[nocompress]{cite}
\usepackage[noblocks]{authblk} 
\allowdisplaybreaks[3]

\usepackage[T1]{fontenc}
\usepackage[utf8]{inputenc}
\usepackage{graphicx}
\usepackage[font=small]{caption}
\usepackage[font=scriptsize]{subcaption}
\usepackage[english]{babel}

\newcommand{\tmop}[1]{\ensuremath{\operatorname{#1}}}

\newtheorem{proposition}{Proposition}
\newtheorem{theorem}{Theorem}

\theoremstyle{definition}
\newtheorem{definition}{Definition}
\newtheorem{example}{Example}
\newtheorem{assumption}{Assumption}

\usepackage[margin=1in]{geometry}
\usepackage[colorlinks=false,hypertexnames=false]{hyperref} 

\usepackage{mathtools}
\usepackage{autonum}
\makeatletter
\newcommand{\restore@Environment}[1]{%
  \AtBeginDocument{%
  	\csletcs{#1*}{#1}%
  	\csletcs{end#1*}{end#1}%
  }%
}
\forcsvlist\restore@Environment{alignat,equation,gather,multline,flalign,align} 
\makeatother

\usepackage{algpseudocode}
\usepackage{algorithm}
\newcommand\Algphase[1]{%
	\vspace*{-.7\baselineskip}\Statex\hspace*{\dimexpr-\algorithmicindent-2pt\relax}\rule{\textwidth}{0.4pt}%
	\Statex\hspace*{-\algorithmicindent}\textbf{#1}%
	\vspace*{-.7\baselineskip}\Statex\hspace*{\dimexpr-\algorithmicindent-2pt\relax}\rule{\textwidth}{0.4pt}%
}

\newcommand{\la}{\leftarrow}
\newcommand{\mb}{model-based regime}

\def\nhf{\nabla \hat{f}}
\newcommand{\etamin}{\eta_{\min}}
\newcommand{\etamax}{\eta_{\max}}

\usepackage{xcolor}
\definecolor{darkgreen}{rgb}{0,0.5,0}

\begin{document}

\title{Accelerating Model-Free Policy Optimization Using Model-Based Gradient: A Composite Optimization Perspective}

\author{Yansong Li\thanks{Email: \texttt{yli340@uic.edu}}}
\author[1]{Shuo Han}
\affil[1]{University of Illinois at Chicago}

\maketitle

\begin{abstract}
    We develop an algorithm that combines model-based and model-free methods for solving
    a nonlinear optimal control problem with a quadratic cost in which the system
    model is given by a linear state-space model with a small additive
    nonlinear perturbation. We decompose the cost into a sum of two functions, one having an explicit form obtained from the approximate linear model, the other being a black-box model representing the unknown modeling error. The decomposition allows us to formulate
    the problem as a composite optimization problem. To solve the optimization
    problem, our algorithm performs gradient descent using the gradient obtained from the approximate linear model
    until backtracking line search fails, upon which the model-based gradient is compared with the exact gradient obtained from a model-free algorithm. The difference between the model gradient and the exact gradient is then used for compensating future gradient-based updates.
    Our algorithm is shown to decrease the number of function evaluations compared with traditional
    model-free methods both in theory and in practice.
\end{abstract}

\section{Introduction}

Learning-based algorithms have flourished in the field of control~\cite{recht_tour_2019} in recent years with applications in robotic manipulation and locomotion~\cite{kohl_policy_2004,levine_end--end_2016}, energy systems~\cite{chen_reinforcement_2021}, and transportation~\cite{wu_flow_2017}.  Learning-based control algorithms can be classified into two categories: model-free algorithms and model-based algorithms. Model-free algorithms do not utilize an explicit form of the
dynamics and can handle dynamics that are hard to model. In the meantime, model-free algorithms
require a large amount of data to learn an optimal policy. This may
become impractical for physical systems, in which data collection is often
expensive since every collection involves physical
interactions with the environment.

In contrast, model-based algorithms maintain an explicit form of the
dynamics such as a state-space model. With the aid of a model, model-based algorithms
usually require less amount of data~\cite{tu_gap_2019-1} to learn an optimal
policy. However, model-based algorithms tend to have difficulties in modeling complex dynamics and may suffer from model bias~\cite{deisenroth_gaussian_2015,atkeson_comparison_1997,abbeel_using_2006} when the model class is not sufficiently rich.

In this paper, we develop a method that combines model-free and
model-based methods to learn an optimal policy for controlling a nonlinear system under a
quadratic cost. We formulate the optimal control
problem as a composite optimization problem, in which the cost is expressed as a composite function
defined by the sum of a model-based part and a model-free part: the model-based part has an analytical form, whereas the model-free part is viewed as a black box.  We then develop a hybrid
gradient-based algorithm that searches for the optimal solution using the \emph{model gradient}, i.e., the gradient
of the model-based part, plus a corrective compensation term, where the compensation is intermittently updated by the
gradient of the model-free part. The algorithm is shown to require fewer function
evaluations compared with purely model-free algorithms both in theory and in practice.

\paragraph{Related work}

There have been several attempts in combining model-based and model-free
algorithms. Qu et al.~\cite{qu_combining_2020} considered a modification of
the linear quadratic regulator (LQR) problem, where a small additive nonlinear
perturbation is introduced to the original linear dynamics. They first applied
model-based control using the approximate model given by the linear part of
the dynamics to obtain a near-optimal policy. Then, they showed that, using
the near-optimal policy as the initial policy, the model-free policy gradient
method is able to produce an optimal policy for the modified LQR problem.
Chebotar et al.~\cite{chebotar_combining_2017-1} developed an algorithm to learn
time-varying linear-Gaussian controllers for reinforcement learning. They
proved that the update of a model-free algorithm, policy improvement with path
integrals (PI$^2$), can be decomposed into two components, one depending on an
approximate cost and another depending on the cost residual. 
The two component are updated separately, where the first component is updated with
a model-based algorithm named LQR-FLM, and the second with PI$^2$. Shashua et al.~\cite{shashua_sim_2021-1} used a descent algorithm with a gradient
mapping that is updated in each iteration based on the model and interactions with
the environment. Our algorithm is closest to the one introduced by Abbeel et al.~\cite{abbeel_using_2006}. They used inaccurate
dynamics to train a near-optimal policy, which was used to collect trajectories from the real Markov decision process (MDP). For each state-action pair from the collected trajectories, they compared the next state in the trajectories with the one predicted by the inaccurate MDP dynamics and used the difference to update the inaccurate MDP dynamics by adding a bias term. In comparison, our algorithm focuses on the control cost rather than the system dynamics. Specifically, our algorithm decomposes the control cost into two parts, one of which, called the model-based part, is related to inexact system dynamics. Instead of updating the inaccurate system dynamics, our algorithm updates the gradient of the model-based part of the cost using information from the true dynamics when performing gradient-based policy optimization.

Our work is also related to policy gradient for LQR, discussed in~\cite{fazel_global_2018,bu_lqr_2019}, in that we also use the same zeroth-order method to compute the gradient of the cost. See~\cite{conn_introduction_2009} for a comprehensive coverage of zeroth-order methods in optimization.

Finally, our work is closely related to composite optimization (see footnote\footnote{Throughout this paper, the term composite optimization refers to \emph{additive} composite optimization, in which the objective function can be expressed as a sum of two functions~\cite{nesterov_gradient_2013}. The same term sometimes also refers to the case in which the objective function is a composition of two functions (see, e.g.,~\cite{lewis_proximal_2016}).}). A classical method for solving composite optimization problems is the proximal gradient method~\cite{rockafellar_monotone_1976,beck_fast_2009,parikh_proximal_2014,nesterov_gradient_2013}. The method assumes that the objective function can be expressed as the sum of a differentiable convex function and another possibly non-differentiable convex function. Moreover, the non-differentiable convex function is assumed to be associated with a proximal operator that can be efficiently evaluated. In contrast, our algorithm deals with composite optimization problems in which both functions in the sum are differentiable but does not assume the existence of an efficiently computable proximal operator.

\section{Background: Nonlinear optimal control with quadratic
cost}\label{sec2:Background: Nonlinear optimal control with quadratic
cost}

We consider a modified discrete-time LQR problem, which appeared in~\cite{qu_combining_2020}. The dynamical system is described by a nonlinear
state-space model with state $x_t \in \mathbb{R}^n$ and control input $u_t \in
\mathbb{R}^p$,
\begin{equation}
  x_{t + 1} = A x_t + B u_t + h (x_t,u_t), \label{dyn}
\end{equation}
where $A \in \mathbb{R}^{n \times n}$, $B \in \mathbb{R}^{n \times p}$ and $h
\colon \mathbb{R}^n \times \mathbb{R}^p \rightarrow \mathbb{R}^n$ with $h(0,0) = 0$. The nonlinear
function $h$ is not assumed to admit an explicit form but is considered ``small,'' where the precise meaning of ``small'' will be discussed in Section~\ref{subsec: Progress in the model-based regime}. The goal is to find a linear state feedback controller $u_t = - K x_t$ (see footnote\footnote{The goal is not to find an optimal state feedback controller, which may be nonlinear for nonlinear systems even when the cost is quadratic.}) that minimizes the cost
$C \colon \mathbb{R}^{p \times n} \rightarrow \mathbb{R}$ defined by
\begin{equation}
  C (K) = 
  \mathbb{E}_{x_0 \sim \mathcal{D}} \left( \overset{\infty}{\underset{t =
  0}{\sum}} x_t^{\top} Q x_t + u_t^{\top} R u_t \right), \label{LQRcost}
\end{equation}
where $Q \in \mathbb{R}^{n \times n}$ and $R \in \mathbb{R}^{p \times p}$ are
positive definite and $x_0$ is the initial state and is drawn from a fixed
distribution $\mathcal{D}$. Multiple problems can be formulated as
\eqref{dyn}. For example, the dynamics of a robotic manipulator can be
described by a linear state-space model plus an error term $h$. Note that if
$h$ is zero everywhere, i.e., $h \equiv 0$, the problem is the same as the
vanilla LQR problem, which has a closed-form solution.

\begin{definition}
  \label{hatc}We define $\hat{C}$ as the quadratic cost~\eqref{LQRcost} when
  $h \equiv 0$ in~\eqref{dyn}. We denote the optimal point that minimizes
  $\hat{C}$ by ${\hat{K}^{\star} }$ and the optimal point that minimizes $C$ by $K^{\star}$.
\end{definition}

Throughout this paper, we propose to minimize $C$ using a policy gradient 
method that uses gradient computed from a 
\emph{zeroth-order method}~\cite{fazel_global_2018}, which only requires access to the values of the cost rather than an analytical expression of the gradient mapping. The method is model-free since it does not rely on an explicit form of the system dynamics. Because each evaluation of the cost function requires collecting trajectories generated from the dynamical system, model-free policy gradient methods need to sample numerous trajectories in order to find a
near-optimal policy~\cite{tu_gap_2019-1}. This issue of high sample complexity has been recognized as a major bottleneck in applying model-free policy gradient methods to physical systems, where the collection of trajectories are often expensive. 

\section{Proposed framework}\label{sec2}

To reduce the sample complexity of model-free policy gradient methods, we
propose a policy optimization framework that leverages an inexact model of the
system, given by the linear part of the dynamics. Our key idea is to
reformulate the original policy optimization problem as a composite
optimization problem that explicitly shows how model information enters policy
optimization. The revelation of the role of the model naturally leads to an
optimization algorithm that solves the composite optimization problem with
fewer function evaluations than traditional model-free policy gradient
methods.

\subsection{A composite optimization perspective of policy optimization}

We decompose the cost function $C$ into two components as
\begin{equation}
  C (K) = \hat{C} (K) + r (K), \label{decop}
\end{equation}
where $\hat{C}$ is defined in Definition~\ref{hatc}, and $r$ is the
residual term defined by $r \triangleq C - \hat{C}$. The problem of minimizing
the cost in~\eqref{decop} is an unconstrained optimization problem of the
following form:
\begin{equation}
  \underset{x \in \mathbb{R}^n}{\tmop{minimize}} \quad f (x), \label{opt}
\end{equation}
where $f = \hat{f} + r$. We call $\nabla f(x)$ the \emph{exact gradient}
and $\nabla \hat{f}(x)$ the \emph{model gradient} at $x$. Denote the optimal value
of~\eqref{opt} by $f^{\star}$ and an optimal solution of~\eqref{opt} by
$x^{\star}$. Also, denote $\hat{x}^{\star}$ as a point where $\nabla \hat{f}
(\hat{x}^{\star}) = 0$.

We make the following assumptions throughout the paper.

\begin{assumption}\label{Ass:PL_condition}
  The function $f$ satisfies the PL-condition, i.e., there exists
  some $\mu > 0$ such that
  \begin{equation}
    \frac{1}{2} \| \nabla f (x) \|^2 \geq \mu (f (x) - f^{\star})
    \label{pl}
  \end{equation}
  for all $x$.
\end{assumption}

\begin{assumption}\label{Ass:L_r}
  The residual mapping $r$ is $L_r$-smooth, i.e., $\nabla r$ is $L_r$-Lipschitz continuous.
\end{assumption}

\begin{assumption}\label{Ass:func_evas}
  The gradient value $\nabla f (x)$ can be computed
  with $ n $ function evaluations, where $ n $ is the dimension of 
  $ x $. The model gradient $\nabla \hat{f}$ has a closed-form
  expression.
\end{assumption}

For the modified LQR problem defined by~\eqref{dyn} and~\eqref{LQRcost},
Assumption~\ref{Ass:PL_condition} is satisfied when $K$ is close to $K^{\star}$, since
$ C $  is locally strongly convex near $K^{\star}$~\cite{qu_combining_2020}, and strong convexity implies the PL-condition~\cite{frasconi_linear_2016}.
Assumption~\ref{Ass:L_r} is also satisfied when $K$ is close to
$K^{\star}$ due to the local smoothness of $ C $~\cite{qu_combining_2020}. Assumption~\ref{Ass:func_evas} is satisfied since $\nabla C (K)$ can
be computed by zeroth-order methods, and $\nabla \hat{C} (K)$ has a closed-form
solution~\cite{fazel_global_2018}. 


\subsection{A model-exploiting composite optimization algorithm}

Throughout this paper, we focus on descent methods for solving the
composite optimization problem~\eqref{opt}. At each iteration, the value of
the optimization variable is updated from $x$ to $x_+$ by
\[ x_+ = x - t \Delta (x), \label{uprule} \]
where $t$ is the step size, and $\Delta (x)$ is a descent direction satisfying $f (x
- t \Delta (x)) < f (x)$ when $t$ is sufficiently small.

To solve the problem in~\eqref{opt}, Qu et al.~\cite{qu_combining_2020} used the gradient descent algorithm with the exact gradient starting from $x = \hat{x}^{\star}$. The
algorithm in~{\cite{qu_combining_2020}} can be viewed as applying the gradient descent
algorithm with the model gradient mapping $\nabla \hat{f}$ until $\nabla \hat{f}
(x) = 0$, i.e., $x = \hat{x}^{\star}$, after which the algorithm switches to
using the exact gradient $\nabla f$. In our algorithm, shown in Algorithm~\ref{algo1},
we also use $\nabla \hat{f}$, but we update $\nabla \hat{f}$ via a constant
compensation $\delta$, which is defined as follows. 

\begin{definition}
  Consider a composite function $f = \hat{f} + r$. We call 
  \begin{equation}
    \tilde{g} \triangleq \nabla \hat{f} + \delta \label{up}
  \end{equation}
the \emph{compensated model gradient mapping}  compensated at point $\tilde{x}$, where the compensation $\delta$ at $\tilde{x}$ is a constant defined by
  \begin{equation}
    \delta \triangleq \nabla f (\tilde{x}) - \nabla \hat{f} (\tilde{x}) = \nabla r
    (\tilde{x}) . \label{comp}
  \end{equation}
\end{definition}

Note that the model gradient mapping $ \nhf $ can be viewed as a compensated model gradient mapping $ \tilde{g} $ with zero compensation.
Due to the smoothness of $r$ (Assumption~\ref{Ass:L_r}), when $\tilde{x}$ is close to $x$, $\nabla
r (\tilde{x})$ will be close to $\nabla r (x)$. Thus, the compensated model
gradient $\tilde{g} (x) = \nabla \hat{f}(x) + \nabla r(\tilde{x})$ will be close to $\nabla f (x) = \nabla \hat{f}(x) + \nabla r(x)$. By Assumption~\ref{Ass:func_evas}, the constant compensation $ \delta \triangleq \nabla f (\tilde{x}) - \nabla \hat{f} (\tilde{x}) =  \nabla r(\tilde{x})$ can be computed with $ n $ function evaluations.

\begin{algorithm}
	\caption{Gradient compensation algorithm}
	\label{algo1}
	\begin{algorithmic}[1]
		\State Obtain $ \eta $ by applying backtracking line search
    with the compensated model gradient mapping $ \tilde{g} $ at $ x $, i.e., $ \tilde{g}(x) $.
		\If{backtracking line search succeeds at $x$}
		\Comment{Model-based Regime} 

		\State $ x \gets x - \eta \tilde{g}(x) $ 
		\Else
		\Comment{Model-free Regime} 
		\State Apply backtracking line search with the exact gradient $ \nabla f(x) $ to obtain $ \eta $.

		\State $ x \gets x - \eta \nabla f(x) $
		\State $ \tilde{g} \gets \nhf + \nabla f(x)- \nhf(x)  $ 
		\EndIf

	\end{algorithmic}

\end{algorithm}	

The difference in the use
of model information between our algorithm and the algorithm in~\cite{qu_combining_2020} is highlighted below (where ``LSF'' stands for ``backtracking line search fails''):
\begin{itemize}
  \item Qu et al.~\cite{qu_combining_2020}: $\nabla \hat{f} \xrightarrow{\nabla
  \hat{f} = 0} \nabla f$
  
  \item Our algorithm: $\nabla \hat{f} \xrightarrow{\nabla \hat{f} =
  0} \nabla f \rightarrow \tilde{g}^{(1)}
  \xrightarrow{\tmop{LSF}} \nabla f \rightarrow
  \tilde{g}^{(2)} \xrightarrow{\tmop{LSF}} \cdots$
\end{itemize}
The mapping $\tilde{g}^{(i)}$ is the $i$-th compensated model gradient mapping. Our algorithm performs gradient descent using $\tilde{g}^{(i)}$ until $\tilde{g}^{(i)} (x)$ is no longer a descent direction at the current iterate $x$, which can be discovered from the failure of backtracking line search~\cite{boyd_convex_2004}, i.e.,
\begin{equation}
f (x - \eta \tilde{g} (x)) \geq f (x) - \alpha \eta \| \tilde{g} (x)
\|^2 \label{LSF}
\end{equation}
for any step size $\eta > 0$ and $\alpha \in (0, 1 / 2)$. When backtracking line search fails, our algorithm performs one step of gradient descent using the exact gradient $\nabla f(x)$, replaces $\tilde{g}^{(i)}$ with
$\tilde{g}^{(i + 1)} \triangleq \nabla \hat{f} + \nabla f (x) - \nabla \hat{f} (x)$, and continues with the new compensated model gradient mapping $\tilde{g}^{(i + 1)}$. 
Note that the standard backtracking line search~\cite{boyd_convex_2004} will
not terminate when it fails to find a positive step size $\eta$. We will revisit this issue in Section~\ref{Sec: Modified Backtracking Line Search}.

In the following, we say that the algorithm operates in the \emph{model-free regime} when the exact gradient is used and in the
\emph{model-based regime} when the compensated model gradient mapping $ \tilde{g} $ is used. Let the superscript $i$ denote the index for the model-free iterations and the subscript $j$
the index for the model-based iterations between two model-free iterations, i.e.,
\begin{equation}
  {x_{j + 1}^{(i)}} =  x_j^{(i)} - \eta \tilde{g}^{(i)} (x_j^{(i)})\qquad \text{(model-based update)}
\end{equation}
and
\begin{equation}
  x^{(i + 1)}_1 = x^{(i)}_{N_i} - \eta \nabla f (x^{(i)}_{N_i})\qquad \text{(model-free update)},
\end{equation}
where $N_i$ is the number of model-based iterations between the $(i-1)$-th and the $i$-th model-free iterations. We will drop the iteration index $i$ when the results do not depend on a specific model-free iteration.

\section{Main result}
When the exact gradient $\nabla f$ is used for computing the update $x_+$, our algorithm is identical to the vanilla gradient descent. Therefore, we only need to consider the case when $x_+$ is computed using the
compensated model gradient mapping $\tilde{g}$ for analyzing the performance of our algorithm. In other words, we only need to consider the progress that our algorithm made in the model-based regime.

\subsection{Modified backtracking line search} \label{Sec: Modified Backtracking Line Search}
In the model-based regime, the compensated model gradient mapping $ \tilde{g} $ is used in place of the exact gradient mapping $ \nabla f $. As a result, the standard backtracking line search~\cite{boyd_convex_2004} may no longer terminate, as illustrated in the following example.

\begin{example}
  \label{stuckmb}Consider $f (x) = (x - 2)^2$ with a compensated model
  gradient $\tilde{g} (x) = 2 x $. The compensated model gradient mapping $ \tilde{g} $ satisfies $\tilde{g}
  (x) > 0$ when $x \in (0, 2)$. Also, $f$ is monotonically decreasing when $x
  < 2$, i.e., $f (x + \Delta x) \geq f (x)$ for any $\Delta x < 0$ and $x
  < 2$. So $f (x - \eta \tilde{g} (x)) \geq f (x) > f (x) - \alpha \eta
  \| \tilde{g} (x) \|^2$ for any $\alpha > 0$ and $\eta > 0$. In other words, the Armijo
  condition~\cite[page 33]{nocedal_numerical_2006}, which is used for terminating the standard
  backtracking line search, cannot be satisfied for any positive $\eta$, implying that the
  standard backtracking line search~\cite{boyd_convex_2004} never terminates.
\end{example}

To ensure that the backtracking line search subroutine eventually terminates, we set
a lower bound $\eta_{\min}$ for the step size $\eta$. The algorithm quits the
backtracking line search subroutine and reports that line search fails
whenever $\eta \leq \eta_{\min}$. The modified backtracking line search algorithm is
presented as Algorithm~\ref{BtLS}.

We denote the initial step size in the backtracking line search as $ \etamax $.
Upon successful termination, the modified backtracking line search
always returns a bounded step size $ \eta \in (\etamin,\etamax)$. 
The modified backtracking line search will only be used in the model-based regime, whereas the standard backtracking line search will be used in the model-free regime. In the following sections, we will simply use the term backtracking line search when the type of backtracking line search is clear from the context.

\begin{algorithm}
	\caption{\textsc{BtLineSearch}}
	\label{BtLS}
	\begin{algorithmic}[1]
		\Require $ x,f,\Delta, \alpha, \beta, \eta_{\min}, \etamax $
		\Ensure $ \eta $
		\State $\eta \gets \eta_{\max}$
		\While{$ f(x - \eta\Delta(x)) > f(x) - \alpha \eta \| \Delta(x)\|^2 $} \label{alg:cond:armijo}
		\If{$\eta > \eta_{\min}$}
		\State $\eta \gets \beta \eta$
		\Else
		\State $ \eta \gets 0 $ \Comment{Line search fails.}
		\EndIf
		\EndWhile
	\end{algorithmic}
\end{algorithm}

Another issue we need to deal with, which will be shown in the next example, is that the algorithm may get stuck in the
model-based regime and never converge to the optimal solution.

\begin{example}
  \label{stuckinmb}Consider the case in Example~\ref{stuckmb}. The update rule
  for this case is given by $x_{j + 1} = x_j - \eta_j \tilde{g} (x_j) = (1 - 2 \eta_j) x_j$
  for $j = 1, \ldots, N$. When the update starts at $x_1 < 0$ with $\etamax < 1 /
  2$, we have $x_j \leq (1 - 2 \etamax)^j x_1 < 0$ for all $j$. Fix an $\alpha \in
  (0, 1 / 2)$, and suppose $\eta_{\min} < 1 - \alpha - 2 / x_1$. It can be verified
  that backtracking line search always succeeds since $f (x_{j + 1}) < f (x_j) - \alpha \eta_j \|
  \tilde{g} (x_j) \|$ for $\eta_j \in (\eta_{\min}, 1 / 2)$. Despite the success
  of line search, the algorithm never converges to the optimal solution
  $x^{\star} = 2$.
\end{example}

Example~\ref{stuckinmb} shows that an inappropriate choice of $ \etamax $ may cause the algorithm to converge to a
suboptimal solution. In general, however, it is difficult to verify whether the choice of $\eta_{\max}$ is appropriate.
We will discuss in Section~\ref{subsec:sufficient decrease condition} how to
deal with this non-convergent behavior.

\subsection{The sufficient decrease condition}\label{subsec:sufficient decrease condition}

Example~\ref{stuckinmb} shows that the algorithm may not ``make sufficient
progress'' even when line search succeeds. In the following, we formally
define what ``sufficient progress'' means and give a verifiable
condition to test whether the algorithm makes sufficient progress when line
search succeeds.

\begin{definition}
  A point ${x_+} $ is said to satisfy the \emph{sufficient decrease
  condition} for $f$ at $x$ if
  \begin{equation}
    f (x_+) < f (x) - t \| \nabla f (x) \|^2 \label{suffid}
  \end{equation}
  for some $t > 0$. 
\end{definition}

In vanilla gradient descent, when backtracking line search succeeds, the Armijo condition
\begin{equation}
  f (x - \eta \nabla f (x)) < f (x) - \alpha \eta \| \nabla f (x) \|^2
  \label{vgd}
\end{equation}
is satisfied for some step size $\eta > 0$ and $\alpha \in (0, 1/2)$.
It can be seen from~\eqref{vgd} that ${x_+} 
= x - \eta \nabla f (x)$ satisfies the sufficient decrease condition for $f$ at $x$. Because backtracking line search always succeeds
when $\nabla f (x) \neq 0$, the condition in~\eqref{vgd} always holds, implying that the function
value will decrease sufficiently in every iteration. The sufficient decrease
condition is central to proving the convergence of vanilla
gradient descent~\cite{boyd_convex_2004} as well as our algorithm, which
will be shown in Section~\ref{subsec: Progress in the model-based regime}. The following example shows that backtracking line search may
not guarantee sufficient decrease when the model gradient is used.

\begin{example}
  Consider again the case in Example~\ref{stuckmb}. Set $x = - 2$, $\alpha = 1
  / 2$, and $\eta = 1 / 2$. Backtracking line search succeeds since $f (x - \eta
  \tilde{g} (x)) = 4 < 12 = f (x) - \alpha \eta \| \tilde{g} (x) \|^2$. Nevertheless, the sufficient decrease condition is not satisfied since $f (x - \eta \tilde{g}
  (x)) \geq 0 = f (x) - \alpha \eta \| \nabla f (x) \|^2$.
\end{example}

Although the success of backtracking line search does not in general imply the sufficient decrease condition, the following proposition shows that the implication holds when $\|\nabla f (x) - \tilde{g}(x) \|$ is sufficiently small.

\begin{proposition}
  Suppose backtracking line search succeeds at $x$ when using the compensated
  model gradient mapping $\tilde{g}$, and $\tilde{g}$ satisfies \label{suffi}
  \begin{equation}
    \| \nabla f (x) - \tilde{g} (x) \| \leq \frac{1 - \gamma}{\gamma} \|
    \tilde{g} (x) \| \label{wsc}
  \end{equation}
  for some $\gamma \in (0, 1)$. Then the updated point $x_+ =
  x - \eta \tilde{g} (x)$, where $\eta$ is given by backtracking line search,
  satisfies the sufficient decrease condition for $f$ at $x$.
\end{proposition}

\begin{proof}
  From the triangle inequality and~\eqref{wsc}, one can see that
  \begin{equation}
    \| \nabla f (x) \| - \| \tilde{g} (x) \| \leq \| \nabla f (x) -
    \tilde{g} (x) \| \leq \frac{1 - \gamma}{\gamma} \| \tilde{g} (x) \| .
  \end{equation}
  Since $\gamma > 0$, it holds that
  \begin{equation}
    \gamma (\| \nabla f (x) \| - \| \tilde{g} (x) \|) \leq (1 - \gamma)
    \| \tilde{g} (x) \| .
  \end{equation}
  Thus,
  \begin{equation}\label{equ:compare_norm_of_grad}
    \| \tilde{g} (x) \| \geq \gamma \| \nabla f (x) \| .
  \end{equation}
  Recall that, when the backtracking line search algorithm returns a nonzero step $\eta
  \neq 0$, it holds that
  \begin{equation}\label{eqn:thm1amijo}
    f (x_+) \leq f (x) - \alpha \eta \| \tilde{g} (x) \|^2 .
  \end{equation}
  Combining~\eqref{equ:compare_norm_of_grad} and~\eqref{eqn:thm1amijo}, we have
  \begin{equation}
    f (x_+) \leq f (x) - \alpha \gamma^2 \eta \| \nabla f (x) \|^2,
    \label{pfprop}
  \end{equation}
  from which the proposition follows.
\end{proof}

Define $e (x) \triangleq \| \nabla f (x) - \tilde{g} (x) \|$, the error of
the compensated model gradient mapping $\tilde{g}$ at $x$. According to Proposition~\ref{suffi},
backtracking line search will ensure a sufficient decrease for updates using the
compensated model gradient mapping $ \tilde{g} $ when $e (x)$ is small. Conceivably, one can determine whether the sufficient decrease condition is still met through 
monitoring the value of $e (x)$. However, it is desirable to avoid evaluating $e (x)$
directly since computing $\nabla f$ is expensive.

Instead of evaluating $e(x)$ directly, we propose to construct an upper bounding sequence $\{\bar{e}_j\}$ satisfying
$\bar{e}_j \geq e (x_j)$ for all $j = 1, \ldots, N$ and monitor $\bar{e}_j$ in place of $e(x_j)$. Specifically, we construct
$\{ \bar{e}_j \}$ by choosing
\begin{equation}
  \bar{e}_1 = 0, \qquad \bar{e }_{j + 1} = L_r \| x_{j + 1} - x_j \| +
  \bar{e}_j . \label{cons}
\end{equation}
The following theorem shows that the sequence $\{\bar{e}_j\}$ constructed in~\eqref{cons} can be used to guarantee the
sufficient decrease condition when backtracking line search succeeds.

\begin{theorem}\label{problre}
  Consider a sequence $\{ x_j \}_{j = 1}^N$ defined by $x_{j + 1} =
  x_j - \eta_j \tilde{g} (x_j)$, where $\eta_j$ is given by the
  backtracking line search. Suppose the backtracking line search using $\tilde{g} (x_j)$
  succeeds at $x_j$ for all $j$. Let $\{\bar{e}_j\}$ be defined by~\eqref{cons}.
  If for a fixed
  $\gamma \in (0, 1)$, we have
  \begin{equation}
    \bar{e}_j \leq \frac{1 - \gamma}{\gamma} \| \tilde{g} (x_j) \|
    \label{lre}
  \end{equation}
  for any $j$, then $x_{j + 1}$ satisfies the sufficient decrease condition for
  $f$ at $x_j$ for all $j = 1,\dots,N$. 
\end{theorem}

\begin{proof}
  By the triangle inequality and the smoothness of $r$, we have
  \begin{align}
    e (x_{j + 1}) &=  \| \nabla r (x_{j + 1}) - \delta \| \nonumber\\
    & \leq  \| \nabla  r (x_{j + 1}) - \nabla  r (x_j) \| + \| \nabla  r
    (x_j) - \delta \| \nonumber\\
    & \leq L_r \| x_{j + 1} - x \| + e (x_j) . \nonumber
  \end{align}
  By induction and~\eqref{cons} we have $\bar{e} (x_j)
  \geq e (x_j)$. The theorem follows from Proposition~\ref{suffi}.
\end{proof}

We can augment Algorithm~\ref{algo1} by incorporating condition~\eqref{lre} and use the
compensated model gradient mapping $ \tilde{g} $ to update $x$ only when
\eqref{lre} is satisfied and backtracking line search succeeds in finding a positive step size. The factor
$\gamma$ can be viewed as a hyperparameter of the algorithm; a smaller $\gamma$ will make
\eqref{lre} easier to satisfy thus making our algorithm stay for more iterations in the
\mb. The augmented algorithm is presented as Algorithm~\ref{Algo3:Gradient_Compensation_Algorithm}. As long as $\gamma$ is small enough, condition~\eqref{lre} will always be satisfied when the algorithm stays in the model-based regime at $x_{j}$ unless $\tilde{g}\left(x_{j}\right)=0$. When $\tilde{g}(x_{j})=0$, it implies that $x_{j}$ is a stationary point as predicted by $\tilde{g}$, at which point it becomes necessary to quit the model-based regime to verify whether $x_{j}$ is truly a stationary point that satisfies $\nabla f(x_{j})=0$.

\begin{algorithm}
    \caption{Gradient Compensation Algorithm}
    \label{Algo3:Gradient_Compensation_Algorithm}
    \begin{algorithmic}[1]
            \State Start from $ x^{(1)} = \hat{x}^{\star} $, $ \bar{e}_{0} = 0 $ and $ \delta_1 = \nabla r(\hat{x}^{\star}) $.
            \For{$ i = 1:\texttt{iters} $} 
            \State $x^{(i+1)} \la $ \textsc{ModelBasedDescent}$(x^{(i)},\delta_i,\gamma, \nhf,   \bar{e}_0  )$ 
            \State $\delta_{i+1} \la \nabla r(x^{(i+1)}) = \nabla f(x^{(i+1)}) - \nhf(x^{(i+1)})$
            \State $ \eta \gets  $ \textsc{BtLineSearch}$ (x,f,\nabla f(x^{(i+1)}), \alpha, \beta, 0, \etamax) $ %
            \Comment{Standard backtracking line search}
            \State $x^{(i+1)}_{+} \la x^{(i+1)} - \eta \nabla f(x^{(i+1)})$
            \State $ \bar{e}_0 \gets L_r  \| x^{(i+1)}_{+} - x^{(i+1)}\| $
            \State $ x^{(i+1)} \gets x^{(i+1)}_{+}$ 
            \EndFor
        \end{algorithmic}	
    
        \begin{algorithmic}[1]
          \Algphase{Algorithm \normalfont\textsc{ModelBasedDescent}}
            \Require $ x $, $ \delta $, $ \gamma $, $ \nhf $, $ \eta_{\min} $, $ \alpha $, $ \beta $, $ \bar{e}_0 $ 
            \Ensure $ x $
            \State $ x_{+} \gets x $
            \State $\tilde{g} \gets \nabla\hat{f} + \delta $
            \State $\bar{e} \gets \bar{e}_0$
            \While{ $\bar{e} \leq \frac{1-\gamma}{\gamma} \|\tilde{g}\| $ and \textsc{BtLineSearch}$ (x,f,\tilde{g}, \alpha, \beta, \etamin, \etamax) \neq 0 $} 
            \State $ x \gets x_{+} $
            \State $\tilde{g} \gets \nabla\hat{f} + \delta $
            \State $ \eta \gets  $ \textsc{BtLineSearch}$ (x,f,\tilde{g}, \alpha, \beta, \etamin, \etamax) $ %
            \Comment{Modified backtracking line search}
            \State $ x_{+} \gets x - \eta \tilde{g} $
            \State $ \bar{e} \gets L_r  \| x_{+} - x\| + \bar{e} $
            \EndWhile
        \end{algorithmic}	
    \end{algorithm}

\subsection{Progress in the model-based regime}\label{subsec: Progress in the model-based regime}

Theorem~\ref{problre} shows that under condition~\eqref{lre}, the success
of the backtracking line search implies a sufficient decrease in the
objective function. An immediate consequence of sufficient decrease is that the suboptimality gap in the model-based regime decreases geometrically, as described in the following proposition.

\begin{proposition}
  \label{propmain}Suppose $f$ satisfies Assumption~\ref{Ass:PL_condition}, condition~\eqref{lre} holds, and the backtracking line search using the compensated model gradient mapping
  $\tilde{g}$ succeeds at
  $x_j$ for all $j = 1, 2, \ldots N$. Then
  \begin{equation}
    f (x_{N + 1}) - f^{\star} \leq (1 - 2 \mu \alpha \gamma^2
    \eta_{\min})^{{N } } (f (x_1) - f^{\star}), \label{main}
  \end{equation}
  where $\eta_{\min}$ is the lower bound on the step size, and $N$ is the number of
  iterations in the model-based regime between two model-free iterations.
\end{proposition}

\begin{proof}
  The proof is adopted from the one for~\cite[Theorem 1]{frasconi_linear_2016}.
  By Assumption~\ref{Ass:PL_condition} and~\eqref{pfprop},
  \begin{equation}
    f (x_{j + 1}) - f^{\star} \leq (1 - 2 \mu \alpha \gamma^2 \eta_j)  (f
    (x_j) - f^{\star}) . \label{eq23}
  \end{equation}
  Applying~\eqref{eq23} recursively yields
  \[ f (x_{N + 1}) - f^{\star} \leq \underset{j}{\prod} (1 - 2 \mu
     \alpha \gamma^2 \eta_j)  (f (x_1) - f^{\star}) \leq (1 - 2 \mu
     \alpha \gamma^2 \eta_{\min})^N (f (x_1) - f^{\star}) . \]
\end{proof}

The \emph{relative error reduction} in the model-based regime of our
algorithm, according to Proposition~\ref{propmain}, is given by $(1 - 2 \mu
\alpha \gamma^2 \eta_{\min})^{N} $, where $N$ is the total number of
iterations in the model-based regime. The relative error reduction represents the progress made in the model-based regime. 
Although it seems that a larger $N$ will
lead to a larger error reduction, this is not true in general. To make the algorithm spend more iterations in the model-based regime, it is necessary to decrease either $\etamin$ to prevent the failure of line search or $\gamma$ to make condition~\eqref{lre} easier to satisfy, both of which can potentially make $(1 - 2 \mu \alpha \gamma^2 \eta_{\min})^N $ larger even if $N$ grows.

In Theorem~\ref{thnmin} that will soon follow, we will show that the algorithm can ensure sufficient decrease for at least a few model-based iterations when backtracking line search succeeds. Before presenting the theorem, we need to introduce the following definition. 

\begin{definition}[Bounded update direction]
  \label{def:bounden_update_direction}
  An update direction mapping $\Delta$
  is called a \emph{bounded update direction} if there exist $\kappa_{\min}$ and
  $\kappa_{\max}$ satisfying $\kappa_{\min} \in (0, 1)$ and $\kappa_{\max} >
  \kappa_{\min}$ such that for any $x$ and any bounded step size $\eta \in
  (\eta_{\min}, \eta_{\max})$, it holds that
  \begin{equation}
    \kappa_{\min} \| \Delta (x) \| \leq \| \Delta (x_+) \| \leq
    \kappa_{\max} \| \Delta (x) \|,
  \end{equation}
  where $x_+ = x - \eta \Delta (x)$.
\end{definition}

We give some examples in the following to show how the bounded update direction
is related to our compensated model gradient mapping $ \tilde{g} $.

\begin{example}
  \label{scaex}Consider the quadratic function $f$ defined by $f (x) =
  (1 / 2) (x^2 + c_2 (x - c_3)^2),$ It can be shown that the compensated model
  gradient mapping $\tilde{g}$, given by $\tilde{g} (x) = x + \delta$, is a bounded
  updated direction. This is because  
  \begin{align}
     \tilde{g} (x_+) &= x_+ + \delta = x  - \eta  \tilde{g} (x ) + \delta
    \nonumber\\
     \qquad &= x  - \eta  (x  + \delta) + \delta = (1 - \eta ) (x  + \delta)
    = (1 - \eta ) \tilde{g} (x ) . \nonumber
  \end{align}
  Since $\eta \in (\eta_{\min}, \eta_{\max})$, it suffices to choose $\kappa_{\max} = 1 - \eta_{\min}$ and $\kappa_{\min} = 1 -
  \eta_{\max}$ to establish that $\tilde{g}$ is a bounded updated direction.
\end{example}

\begin{example}
  \label{exquad}Consider the function $f$ defined as 
  \begin{equation}
    f (x) = \frac{c_1}{2} x^{\top} P x + \frac{c_2}{2} (x - c_3)^{\top} Q (x - c_3)
    \label{quad}
  \end{equation}
  where $P$ and $Q$ are positive definite. Without loss of generality, we choose $ c_1 = 1 $. It can be shown that the compensated model gradient mapping   
  $\tilde{g}$, given by $\tilde{g} (x) = P x + \delta$, is a bounded update
  direction: Let $x_+ = x - \eta \tilde{g} (x)$. It holds that $\tilde{g} (x_+) = (1 -
  \eta P) (P x + \delta) = (1 - \eta  P) \tilde{g} (x)$. When $\eta \in (\eta_{\min},
  \eta_{\max})$, one can choose $\kappa_{\max} =
  1 - \eta_{\min} \lambda_{\min} (P)$ and $\kappa_{\min} = 1 - \eta_{\max}
  \lambda_{\max} (P)$, where $\lambda_{\max} (P)$ and $\lambda_{\min} (P)$ are
  the maximum and the minimum eigenvalues of $P$, respectively. 
\end{example}

Intuitively, Algorithm~\ref{Algo3:Gradient_Compensation_Algorithm} will stay in the model-based regime when
the compensated model gradient $ \tilde{g}(x) $ is close to the exact gradient $ \nabla f(x) $. To lower bound the number of iterations the algorithm stays
in the model-based regime, we need to bound the change of the compensated model gradient mapping $ \tilde{g} $ at each iterate, which is captured by the notion of bounded update direction introduced in Definition~\ref{def:bounden_update_direction}. Using the notion of bounded update direction, we present in the following theorem a lower bound on the number of model-based iterations.

\begin{theorem}
  \label{thnmin} Let $ x_1 $ be the initial point when 
  Algorithm~\ref{Algo3:Gradient_Compensation_Algorithm} enters the model-based regime with a compensated model gradient mapping $ \tilde{g} $. Suppose $\tilde{g}$ is a
  bounded updated direction mapping with constants $ \kappa_{\max} $ and $ \kappa_{\min} $, and $ \tilde{g}(x_1) \neq 0 $. 
  Then the success of backtracking line search implies the sufficient decrease condition if $N$ satisfies
  \begin{equation}
    \log | \kappa_{\max}^N - 1 | + N \log  \frac{1}{\kappa_{\min}}
     \leq \log  \frac{1 - \gamma}{\gamma \eta_{\max} L_r}
     + \log | \kappa_{\max} - 1 | \label{nmin1}
  \end{equation}
  when $\kappa_{\max} \neq 1$ and
  \begin{equation}
    \log N + N \log  \frac{1}{\kappa_{\min}}  \leq \log
     \frac{1 - \gamma}{\gamma \eta_{\max} L_r} \label{nmin2}
  \end{equation}
  when $\kappa_{\max} = 1$. 
\end{theorem}

\begin{proof}
  Suppose backtracking line search does not ensure sufficient decrease after
  $N$ steps --- that is,
  \begin{equation}
    \bar{e}_{N + 1} > \frac{1 - \gamma}{\gamma} \| \tilde{g} (x_{N + 1}) \| .
    \label{largerel}
  \end{equation}
  When $\kappa_{\max} \neq 1$,
  \[ \bar{e}_{N + 1} = L_r  \underset{j = 1}{\overset{N}{\sum}} \| x_{j
     + 1} - x_j \|  = L_r  \underset{j = 1}{\overset{N}{\sum}}
     \eta_j \| \tilde{g} (x_j) \|  \leq L_r \eta_{\max} \frac{1 -
     \kappa_{\max}^N}{1 - \kappa_{\max}} \| \tilde{g} (x_1) \|. \]
  Combine with equation~\eqref{largerel} to obtain
  \[ L_r \eta_{\max} \frac{1 - \kappa_{\max}^N}{1 - \kappa_{\max}} \|
     \tilde{g} (x_1) \| > \frac{1 - \gamma}{\gamma} \| \tilde{g} (x_{N + 1})
     \| \geq \frac{1 - \gamma}{\gamma} \kappa_{\min}^N \| \tilde{g} (x_1)
     \| . \]
  Rearranging the above formula yields
  \begin{equation} \label{eq:nmin1_contra}
  \log | \kappa_{\max}^N - 1 | + N \log \frac{1}{\kappa_{\min}}
   > \log  \frac{1 - \gamma}{\gamma \eta_{\max} L_r} +
  \log | \kappa_{\max} - 1 | .
  \end{equation}%
  Similarly, when $\kappa_{\max} = 1$, by the triangle inequality, 
  \[ \bar{e}_{N + 1} = L_r  \underset{j = 1}{\overset{N}{\sum}} \| x_{j
     + 1} - x_j \|  = L_r  \underset{j = 1}{\overset{N}{\sum}}
     \eta_j \| \tilde{g} (x_j) \|  \leq L_r \eta_{\max} N \|
     \tilde{g} (x_1) \| . \]
  Thus, the same procedure as above gives us
  \begin{equation}\label{eq:nmin2_contra}
    \log N + N \log  \frac{1}{\kappa_{\min}}  > \log
    \frac{1 - \gamma}{\gamma \eta_{\max} L_r}.
  \end{equation}
  The theorem follows by taking the contrapositive of the results in~\eqref{eq:nmin1_contra} and~\eqref{eq:nmin2_contra}.
\end{proof}

The left-hand side of~\eqref{nmin1} and that of~\eqref{nmin2} are monotonically
increasing with respect to $N$. Let $N_{\min}$ be the largest $N$ satisfying
\eqref{nmin1} or~\eqref{nmin2}, depending on the value of $\kappa_{\max}$. Theorem~\ref{thnmin} ensures that our algorithm makes sufficient progress
in the model-based regime for at least $N_{\min}$
iterations as long as backtracking line search succeeds. Because the right-hand side of~\eqref{nmin1} and~\eqref{nmin2} are
monotonically decreasing as $ L_r $ grows, a smaller $ L_r $ is desired for making the algorithm stay for more iterations in the model-based regime when backtracking line search succeeds. Recall in Section~\ref{sec2:Background: Nonlinear optimal control with quadratic
cost}, we require $h$ to be ``small'' without formally defining how to quantify the magnitude of $h$. Theorem~\ref{thnmin} indicates that the smoothness constant $ L_r $ of the residual cost $r$ is a useful metric for quantifying the magnitude of $h$. Indeed, when the dynamics in~\eqref{dyn} are nearly linear, the constant $L_r$ will be small, and the algorithm is guaranteed to make better use of model information by staying in the model-based regime for more iterations. In the extreme case of $h \equiv 0$, the residual cost $r \equiv 0$ and hence $L_r = 0$, implying that the algorithm always stays in the model-based regime. 

\subsection{Number of function evaluations}\label{subsec: Number of func eval}

In Section~\ref{subsec: Progress in the model-based regime}, we showed that Algorithm~\ref{Algo3:Gradient_Compensation_Algorithm} will stay in the model-based regime for at least some iterations and make progress when backtracking line search succeeds. We will show in the following that the model-based regime requires fewer function evaluations than the model-free regime for making the same progress when the dimension $n$ of the problem is large.

Recall that, when gradient descent is applied to strongly convex and smooth objectives, the objective values for two successive iterations $i$ and $i+1$ satisfy 
\begin{equation}
f(x_{i+1})-f^{\star}=\rho(f(x_{i})-f^{\star}),\label{eq:linear_conv}
\end{equation}
where $\rho$ is the convergence ratio. After $N$ iterations, the relative error reduction is given by $\phi\triangleq(f(x_{N})-f^{\star})/(f(x_{0})-f^{\star})=\rho^{N}$. We propose to characterize the progress made per iteration by $|\log\phi|/N$, which recovers $|\log\rho|$ and hence is independent of $N$ when~\eqref{eq:linear_conv} holds. Similarly, we define the progress made per function evaluation as $|\log\phi|/(mN)$, where $m$ is the number of function evaluations in each iteration.

By Assumption~\ref{Ass:func_evas}, at any point $x$, the number of function evaluations for computing $\nabla f(x)$ is $n$, whereas the number of function evaluations for computing $\nabla\hat{f}(x)$ is zero. From Algorithm~\ref{BtLS}, the maximum number $m_{\max}$ of function evaluations inside the backtracking line search subroutine is the smallest integer $m$ satisfying $\beta^{m-1}\eta_{\max}\leq\eta_{\min}$. Thus, 
\[
m_{\max}=\left\lceil \frac{\log(\etamin/\etamax)}{\log\beta} + 1 \right\rceil ,
\]
which only depends on the parameters of backtracking line search. Also, from Proposition~\ref{propmain}, the relative error reduction in the model-based regime is $(1-2\mu\alpha\gamma^{2}\eta_{\min})^{N}$. Thus, the progress made per function evaluation in the model-based regime is lower bounded by 
\begin{equation}
\frac{|\log(1-2\mu\alpha\gamma^{2}\eta_{\min})^{N}|}{m_{\max}N}=\frac{|\log(1-2\mu\alpha\gamma^{2}\eta_{\min})|}{m_{\max}},\label{eq:progress made per iteration in the model-based regime}
\end{equation}
which is a constant when the parameters of backtracking line search and $\gamma$ are fixed. In the meantime, when $f$ is assumed to be $L_{f}$-smooth~\cite{frasconi_linear_2016}, the convergence ratio $\rho$ in the model-free regime is $1-\mu/L_{f}$. Also, from Assumption~\ref{Ass:func_evas}, the number of function evaluations in each iteration in the model-free regime is at least $n$, which includes $n$ evaluations needed for computing the exact gradient and any additional evaluations used in backtracking line search. Thus, the progress made per function evaluation is upper bounded by 
\begin{equation}
\frac{|\log \rho|}{n}=\frac{|\log(1-\mu/L_{f})|}{n},\label{eq:progress made per iteration in the model-free regime}
\end{equation}
which decreases with respect to the problem dimension $n$. 

When the dimension $n$ of the problem is large enough, the quantity in \eqref{eq:progress made per iteration in the model-based regime} will be larger than that in~\eqref{eq:progress made per iteration in the model-free regime}, implying that more progress is made per function evaluation in the model-based regime. Since Algorithm~\ref{Algo3:Gradient_Compensation_Algorithm} always makes progress in the model-based regime as discussed in Section~\ref{subsec:sufficient decrease condition}, the algorithm is guaranteed to use fewer function evaluations than a purely model-free method for achieving the same level of suboptimality.

\section{Numerical experiments}

In this section, we shall refer to Algorithm~\ref{Algo3:Gradient_Compensation_Algorithm} as GC, which stands for ``gradient compensation.'' We compare GC with the method in~\cite{qu_combining_2020} in two
different settings: (1) Both $\hat{f}$ and $r$ are convex quadratic functions. ($f = \hat{f} + r$ is
  convex and smooth.) (2) The objective function is given by $C$ in~\eqref{decop}, and $\hat{f}$ is given by $\hat{C}$ from Section~\ref{sec2:Background: Nonlinear optimal control with quadratic
  cost}. ($C$ is non-convex and locally
  smooth.)

\subsection{Quadratic functions}\label{sec:quad_func}

We consider the quadratic functions $ \hat{f} $, $ r \colon \mathbb{R}^{n} \rightarrow \mathbb{R}$ defined as $ \hat{f}(x) =  c_1 x^{\top} P x / 2 $ and $ r(x) = c_2 (x - c_3)^{\top} Q (x -
c_3)/2 $, where both $P$ and $Q$ are positive definite. To ensure that $P$ is positive
definite, we generated its square root $P^{1 / 2}$ randomly, with each entry
drawn independently from the standard Gaussian distribution, after which $P$ was
computed by $P = P^{1 / 2} P^{1 / 2}$. The matrix $Q$ was generated randomly from its spectral decomposition with
a spectral radius of $10$. We chose $ c_1 = 1 $, $c_2 = c_3 =
0.1$, and $ n = 4 $. Thus, the smoothness factor $L_r$ is given by $L_r = c_2 = 0.1$.
The parameters of backtracking line search were given by $\alpha = 0.3$, $\beta = 0.5$, 
$\eta_{\max} = 1$, and $\eta_{\min} = 0.005$.

\paragraph{Results}

\begin{figure}
	\centering{
	\hfill%
	\begin{subfigure}[b]{0.3\textwidth}
		\centering
		\includegraphics[width=\textwidth]{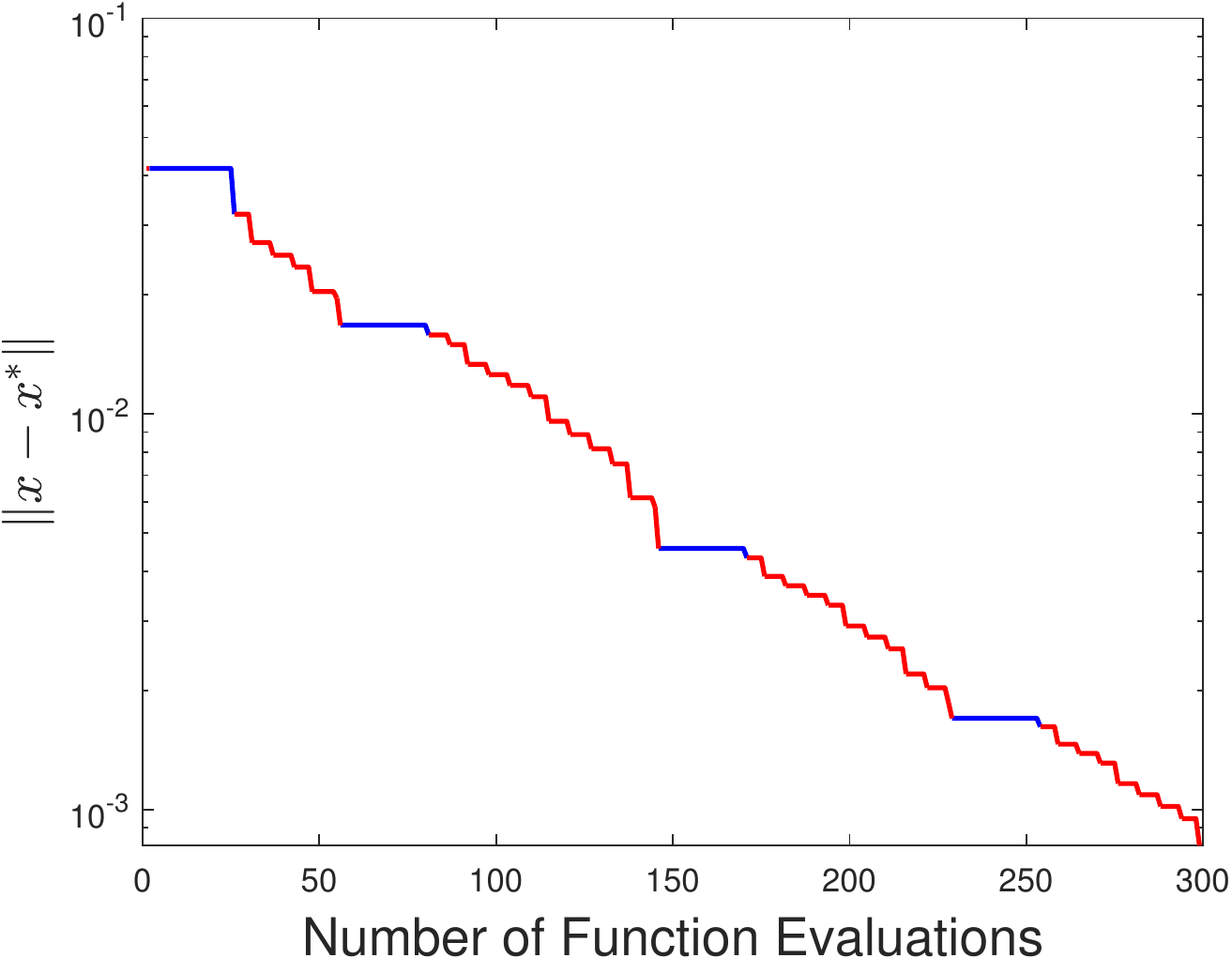}
		\caption{}
		\label{fig:gc_quad}
	\end{subfigure}
	\hfill%
	\begin{subfigure}[b]{0.3\textwidth}
		\centering
		\includegraphics[width=\textwidth]{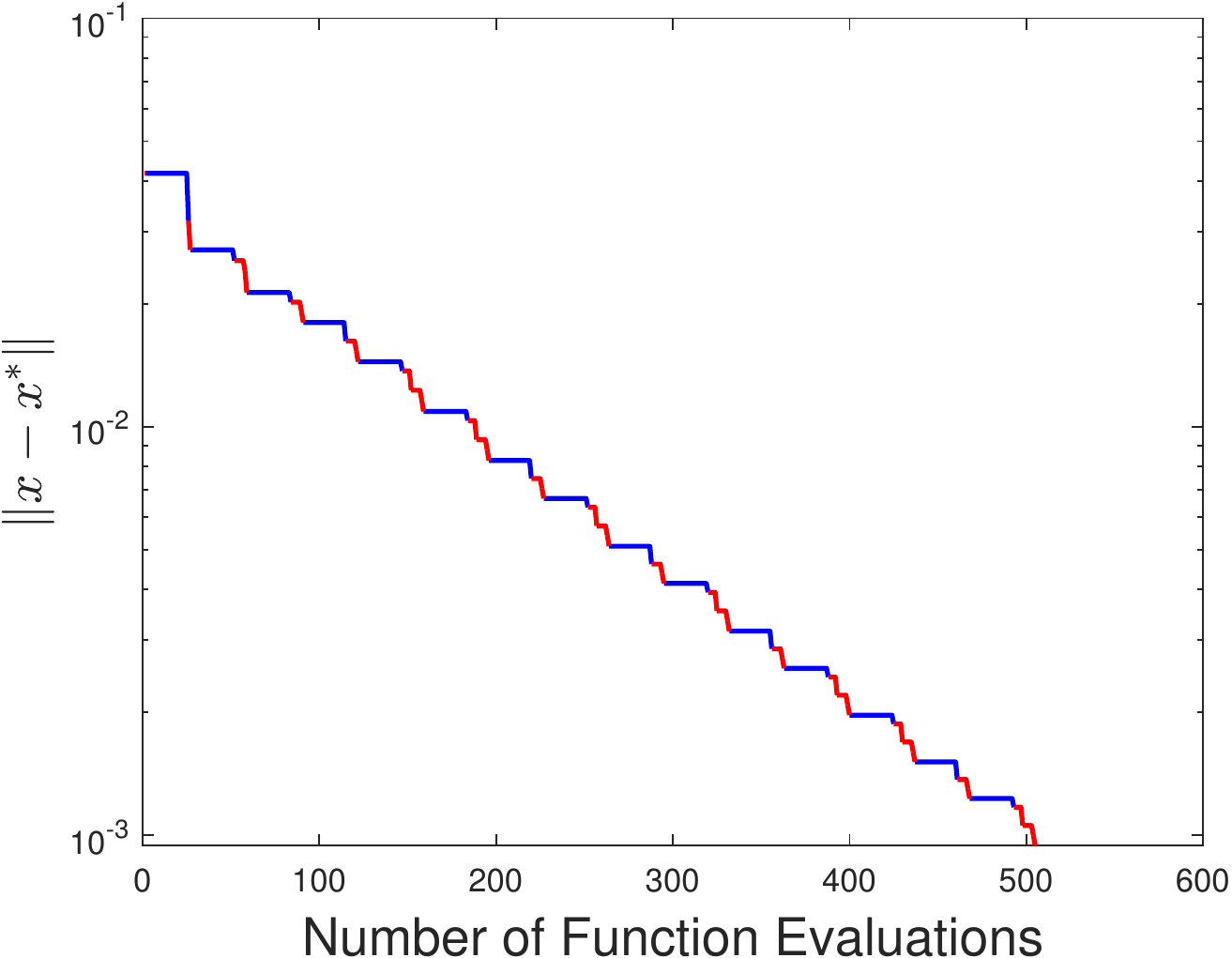}
		\caption{}
		\label{fig:quad_large_gamma}
	\end{subfigure}
	\hfill%
	\begin{subfigure}[b]{0.3\textwidth}
		\centering
		\includegraphics[width=\textwidth]{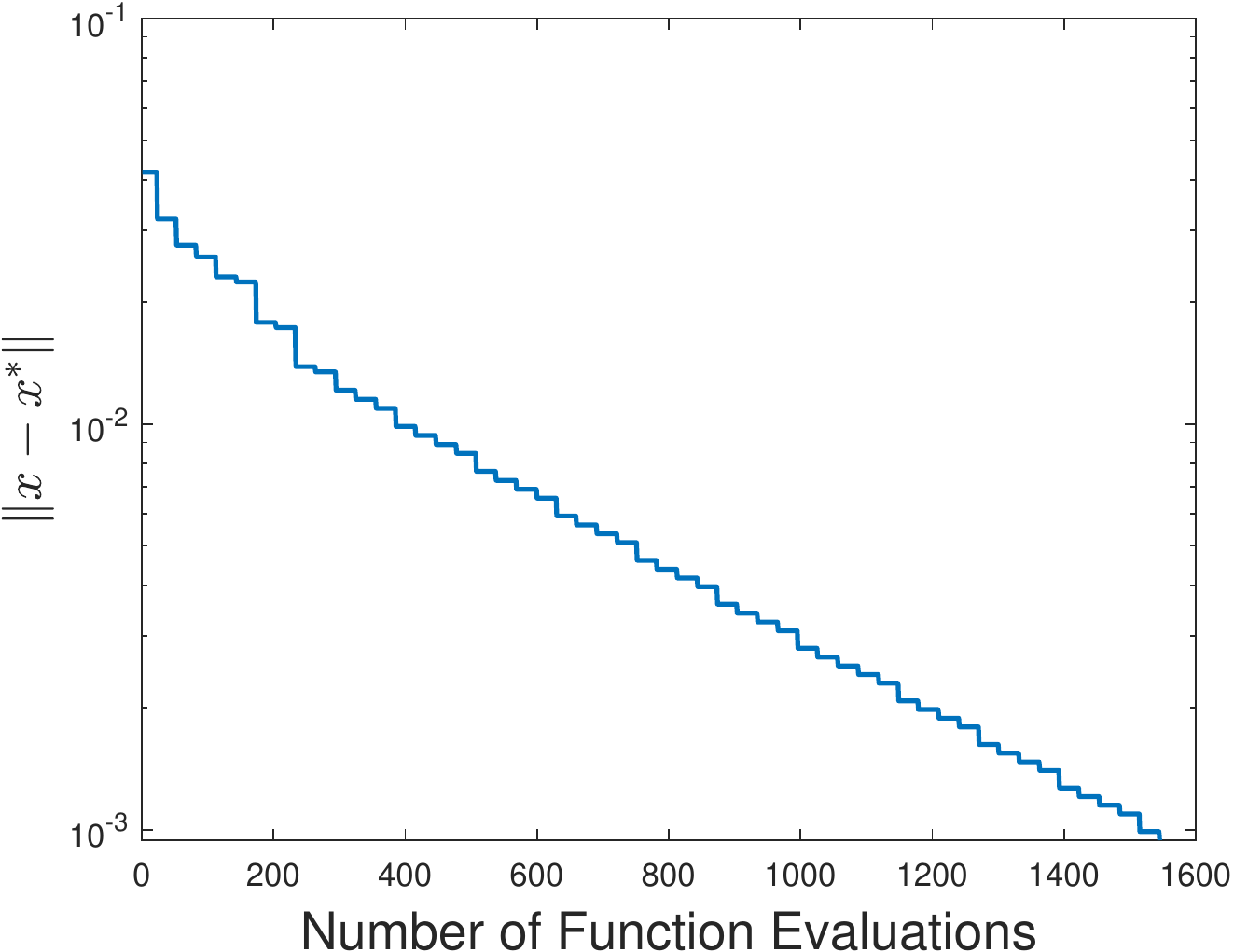}
		\caption{}
		\label{fig:qu_quad}
	\end{subfigure}	
	\hfill%
}
	\caption{Comparison among the convergence of (a) the GC algorithm with $\gamma = 0.6$, (b) the GC algorithm with $\gamma = 0.9$, and (c) the model-free method in~\cite{qu_combining_2020} for quadratic functions. For the GC algorithm, the iterations in the model-free regime are indicated in blue, and the ones in the model-based regime are indicated in red. When $\gamma$ was chosen appropriately, the GC algorithm used fewer function evaluations than the model-free method because it used relatively inexpensive model-based iterations to make progress.}
	\label{fig:quad}
\end{figure}

The results are shown in Fig.~\ref{fig:quad}, where Fig.~\ref{fig:gc_quad} and~\ref{fig:quad_large_gamma} are from the GC algorithm, and Fig.~\ref{fig:qu_quad} is from the method by Qu et al.~\cite{qu_combining_2020}. The blue part of the curve in Fig.~\ref{fig:gc_quad} and ~\ref{fig:quad_large_gamma} represents the model-free regime, and the red represents the
model-based regime. To better reflect the actual computational cost, we used the number of function evaluations rather than the number of iterations to quantify the running time of the algorithms. The vertical axis is the error $ \|x - x^\star \| $, plotted in a logarithmic scale, where $ x^\star $ is the 
optimal solution that minimizes $ f \triangleq \hat{f} + r $. We terminated both algorithms when
$ \|x - x^\star \| < 10^{-3} $. The error $ \|x - x^\star \| $ was only updated upon the completion of backtracking line search. Since each test of the Armijo condition (Line~\ref{alg:cond:armijo} in Algorithm~\ref{BtLS}) in backtracking line search requires one function evaluation, the error $ \|x - x^\star \| $ sometimes remained unchanged over multiple consecutive function evaluations when backtracking line search was in progress.

It can be seen that the GC algorithm required fewer function evaluations for achieving the same level of suboptimality.
Whether the GC algorithm stays longer in the model-based regime or in the model-free regime depends on the choice of $\gamma$ in~\eqref{lre}, which determines whether Algorithm~\ref{Algo3:Gradient_Compensation_Algorithm} should continue to stay in the model-based regime. When a larger $ \gamma = 0.9 $ was used, the algorithm stayed for fewer iterations in the model-based regime because condition~\eqref{lre} became more difficult to satisfy.


\subsection{Modified LQR problem}

Consider the modified LQR problem defined by~\eqref{dyn} and~\eqref{LQRcost} in
Section~\ref{sec2:Background: Nonlinear optimal control with quadratic
cost}. The problem instance is similar to the one used in~\cite{qu_combining_2020}. The matrices $A$ and $B$ were generated randomly,
with each entry drawn from the standard Gaussian distribution. To simplify, the nonlinear
error term $h$ in the dynamics~\eqref{dyn} was set to only depend on the system state $ x $. The simplified error term was defined by $h(x) = \ell x / (1 - 0.9 \sin x)$ with $\ell =
0.01$. The quadratic cost in~\eqref{LQRcost} was defined by $Q = 2 I$ and $R = I$. The dimensions of the problem were set as $n = 4$ and $p = 3$. (In comparison, $n
= p = 2$ was used in~\cite{qu_combining_2020}.) The parameters of backtracking line search were given by $\alpha =
0.3$, $\beta = 0.5$, $\eta_{\max} = 1$, and $\eta_{\min} = 0.05$. 

Unlike the quadratic functions studied in Section~\ref{sec:quad_func}, in the setting of modified LQR, it is difficult to obtain the value of $L_r$.
Instead, we treated $L_r$ as a hyperparameter, performed a number of trials with different choices of $L_r$, and reported the best result.

\subsubsection{Computing the gradient}

The gradient of the cost $\hat{C}$ for the linear system can be computed analytically~\cite{fazel_global_2018} as $\nabla
\hat{C} (K) = 2 E_K \Sigma_K$, where $E_K = (R +
B^{\top} P_K B) K - B^{\top} P_K A$ and $\Sigma_K = \mathbb{E}_{x_0 \sim
	\mathcal{D}} \sum_{t=0}^{\infty} x_t x_t^{\top}$. The matrix $P_K$ in the expression for $E_K$ is the solution to the Lyapunov
equation
\begin{equation}
  A_K^{\top} P_K A_K - P_K + Q + K^{\top} R K = 0,
\end{equation}
where $A_K = A - B K$. For the cost $C$ for the nonlinear system, we computed $\nabla C (K)$ by a zeroth-order method with $C$ approximated by a finite-horizon empirical
cost
\begin{equation}
  C_{\tmop{emp}}^T (K) \triangleq \mathbb{E}_{x_0 \sim
  \mathcal{D}} \left( \overset{T}{\underset{t =
  0}{\sum}} x_{t}^{\top} Q x_{t} + u_{t}^{\top} R u_{t} \right),
\end{equation}
where $\mathcal{D}$ is a uniform distribution over $n$ distinct values $x_0^{(1)}, x_0^{(2)}, \ldots, x_0^{(n)}$. The procedure for computing $\nabla C$ is presented in Algorithm~\ref{alg:Compute_nablaC}. In all the experiments, the time horizon $T$ was set to $50$, and the sampling radius $ r $ was set to $ 0.001 $.

\begin{algorithm}
	\caption{Computing the gradient of $ C $}
	\label{alg:Compute_nablaC}
	\begin{algorithmic}[1]
		\Require $ K $, $ r $, $ n $, $ p $, $ T $, $ A $, $ B $, $ h $, $ Q $, $ R $
		\Ensure $ \nabla C (K) $
		\For{ $ j $ = 1, 2, ..., $ pn $ }
		\State \texttt{//} \textsc{reshape}$(\cdot, p\times n)$ converts a $pn$-dimensional vector to a $p\times n$ matrix.
		\State $ U_j \gets \textsc{reshape}(r\cdot e_j, p\times n)$, where $e_j$ is a unit basis vector whose $j$-th entry is $1$.
		\For{ $ k $ = 1, 2, ..., $ n $ }
		\State  $ x_0 \gets x_{0}^{(k)}$
		\For{$ t  = 0,1,...,T$ } 
 		\State	$ u_{t}\gets-\left(K+U_{j}\right) x_{t} $
		\State  $ x_{t+1} = A x_{t} + B u_{t} + h(x_t) $
		\EndFor
		\State $ C_{j}^{(k)}\gets\sum_{t=0}^{T}x_{t}^{\top} Q x_{t}+u_{t}^{\top} R u_{t} $
		\EndFor
		\State $ C_{j} \gets \frac{1}{n} \sum_{k=1}^{n} C_{j}^{(k)} $ 
		\EndFor
		\State \Return $ \nabla C(K)=\frac{1}{d} \sum_{j=1}^{d} \frac{d}{r^{2}} C_{j} U_{j} \text {, where } d=p n $ 
	\end{algorithmic}
\end{algorithm}

\subsubsection{Results}

\begin{figure}
	\centering{
	\hfill%
	\begin{subfigure}[b]{0.3\textwidth}
		\centering
		\includegraphics[width=\textwidth]{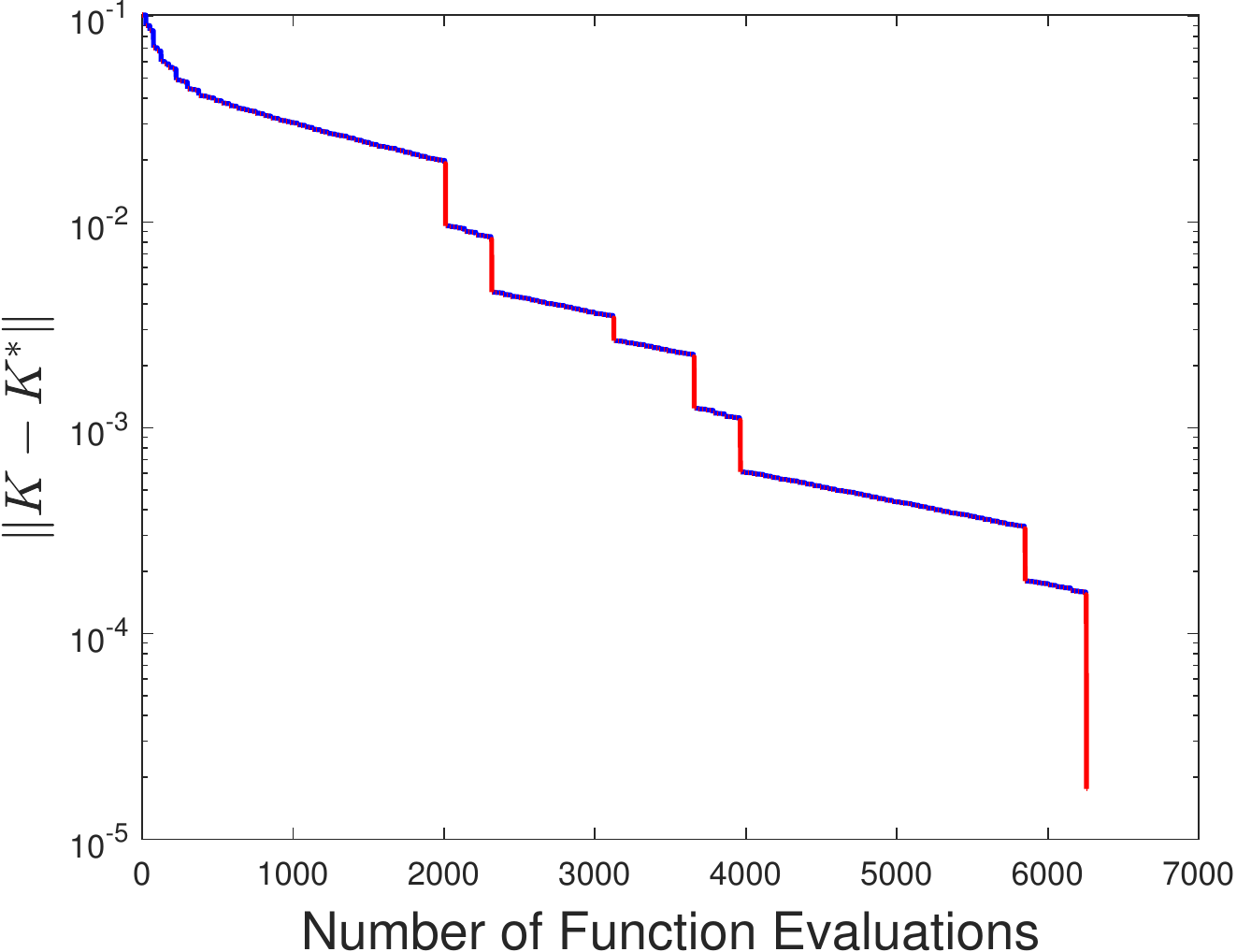}
		\caption{}
		\label{fig:gc_lqr}
	\end{subfigure}
	\hfill%
	\begin{subfigure}[b]{0.3\textwidth}
		\centering
		\includegraphics[width=\textwidth]{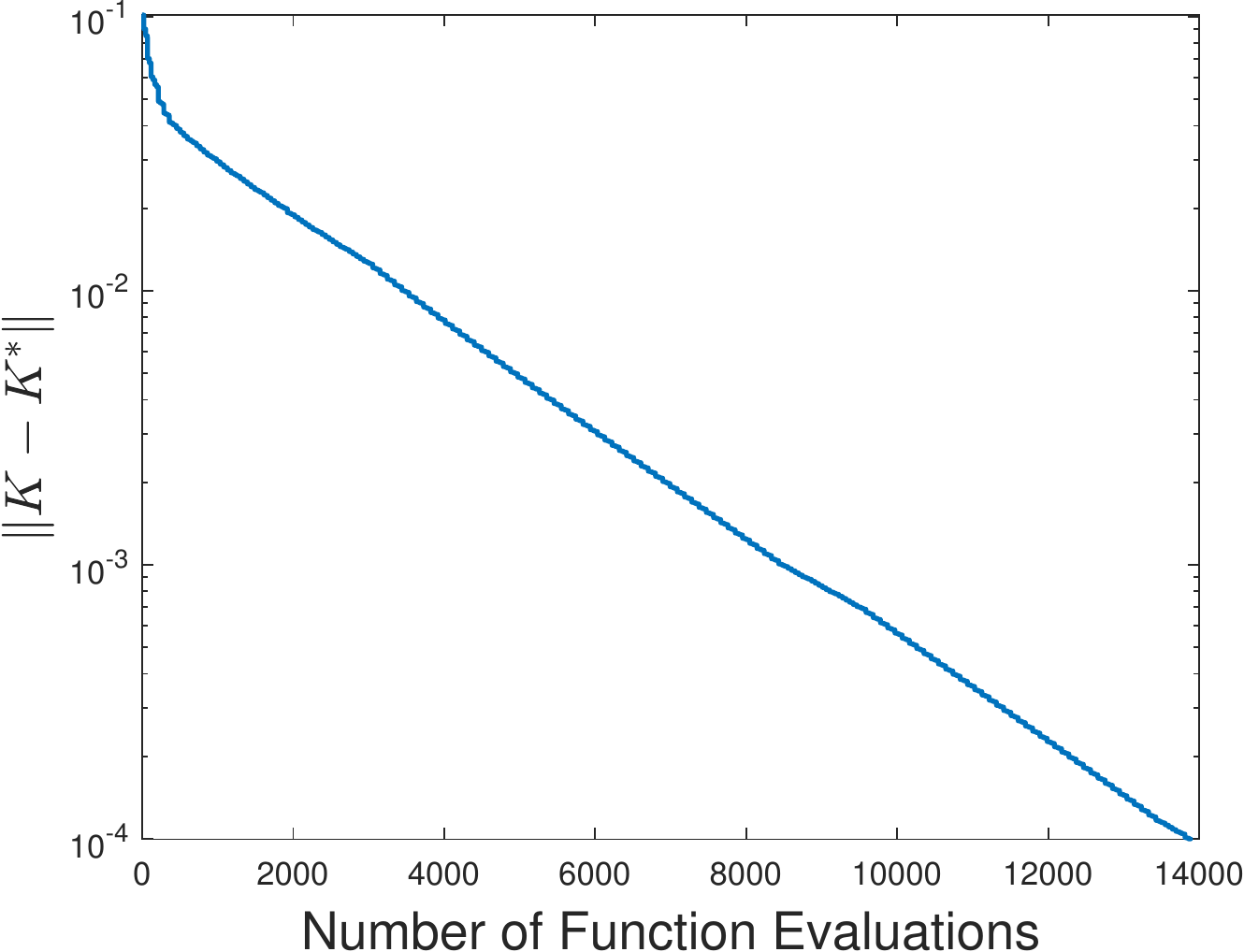}
		\caption{}
		\label{fig:qu_lqr}
	\end{subfigure}	
	\hfill%
}
	\caption{Comparison between the convergence of (a) the GC algorithm with $ \etamin = 0.05 $ and (b) the model-free method in~\cite{qu_combining_2020} for the modified LQR problem. While the GC used fewer function evaluations, it still spent most of the time in the model-free regime, indicated in blue.}
	\label{fig:lqr}
\end{figure}

Fig.~\ref{fig:lqr} compares the result obtained from the GC algorithm and the one from the model-free method by Qu et al.~\cite{qu_combining_2020}. Because the modified LQR does not admit an optimal solution in closed form, the optimal feedback gain $ K^\star $ was computed approximately by running the model-free method in~\cite{qu_combining_2020} until $\|\nabla C({K^\star})\| < 10^{-4}$.
We terminated both algorithms when $ \|K - K^\star \| < 10^{-4} $. 

\paragraph{Effect of $\gamma$ and $L_r$}

The choices of $\gamma$ and $L_r$ affect condition~\eqref{lre}, which is used to monitor whether the sufficient decrease condition holds upon the success of backtracking line search. When the minimum step size $\etamin = 0.05$ was used, both $ \gamma $ and $L_r$ were found to have little effect on the performance of the GC algorithm. 
For instance, as $\gamma$ varied between $ 0.0001 $ to $ 0.9999 $, the convergence of the GC algorithm remained the same, as represented in Fig.~\ref{fig:gc_lqr}. Upon inspecting the numerical simulation results, we found that, for the majority of the time, the GC algorithm was found to leave the model-based regime almost immediately after entering. Moreover, 
the GC algorithm left the model-based regime mostly due to the failure of backtracking line search rather than violating condition~\eqref{lre}. Since neither $\gamma$ nor $L_r$ affects backtracking line search, this explains why changing $\gamma$ and $L_r$ did not help the GC algorithm stay in the model-based regime longer so as to reduce the number of function evaluations.

\paragraph{Effect of $\eta_{\min}$}

\begin{figure}
	\centering{
	\hfill%
	\begin{subfigure}[b]{0.3\textwidth}
		\centering
		\includegraphics[width=\textwidth]{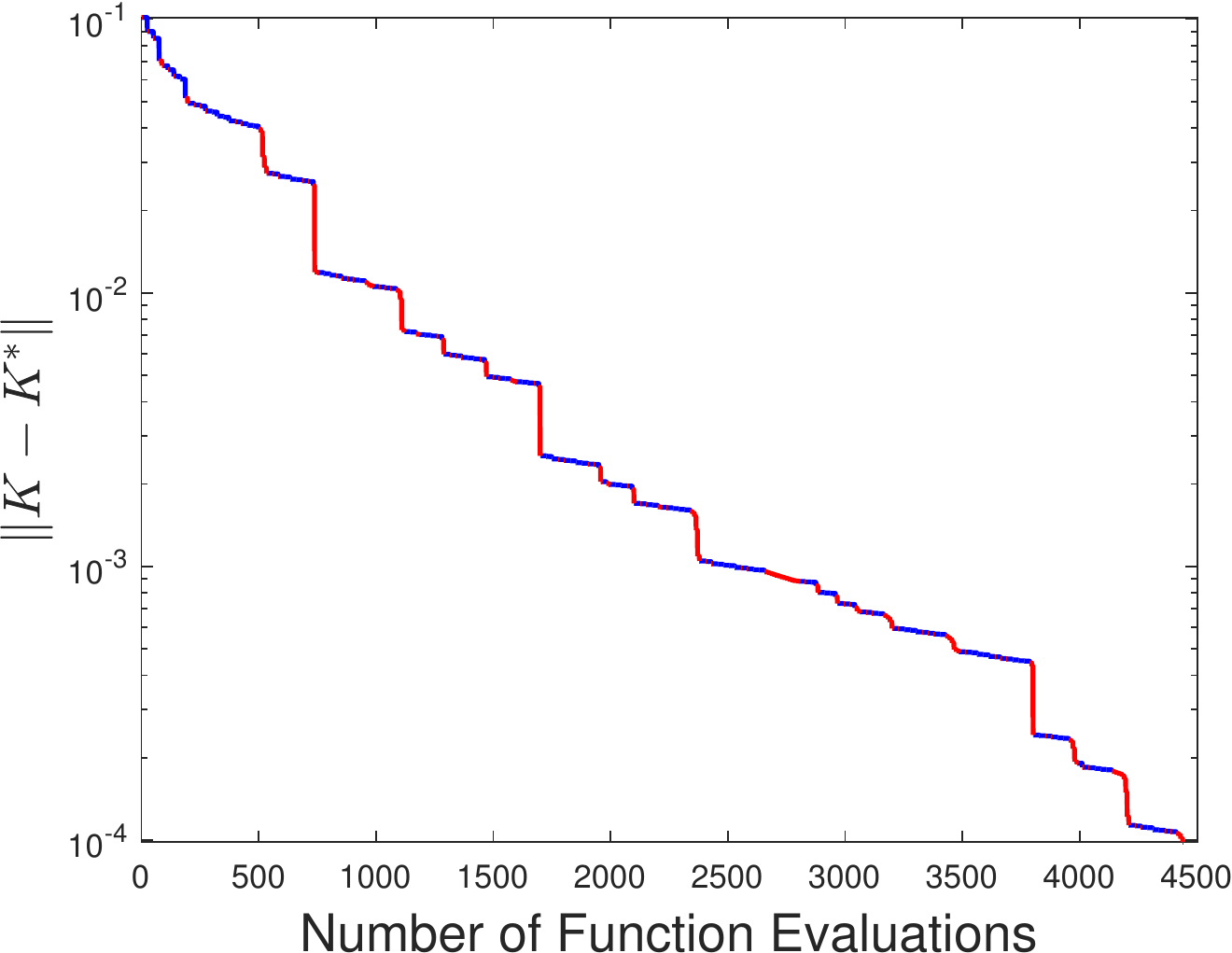}
		\caption{}
		\label{fig:gc_lqr_small_eta_min}
	\end{subfigure}
	\hfill%
	\begin{subfigure}[b]{0.3\textwidth}
		\centering
		\includegraphics[width=\textwidth]{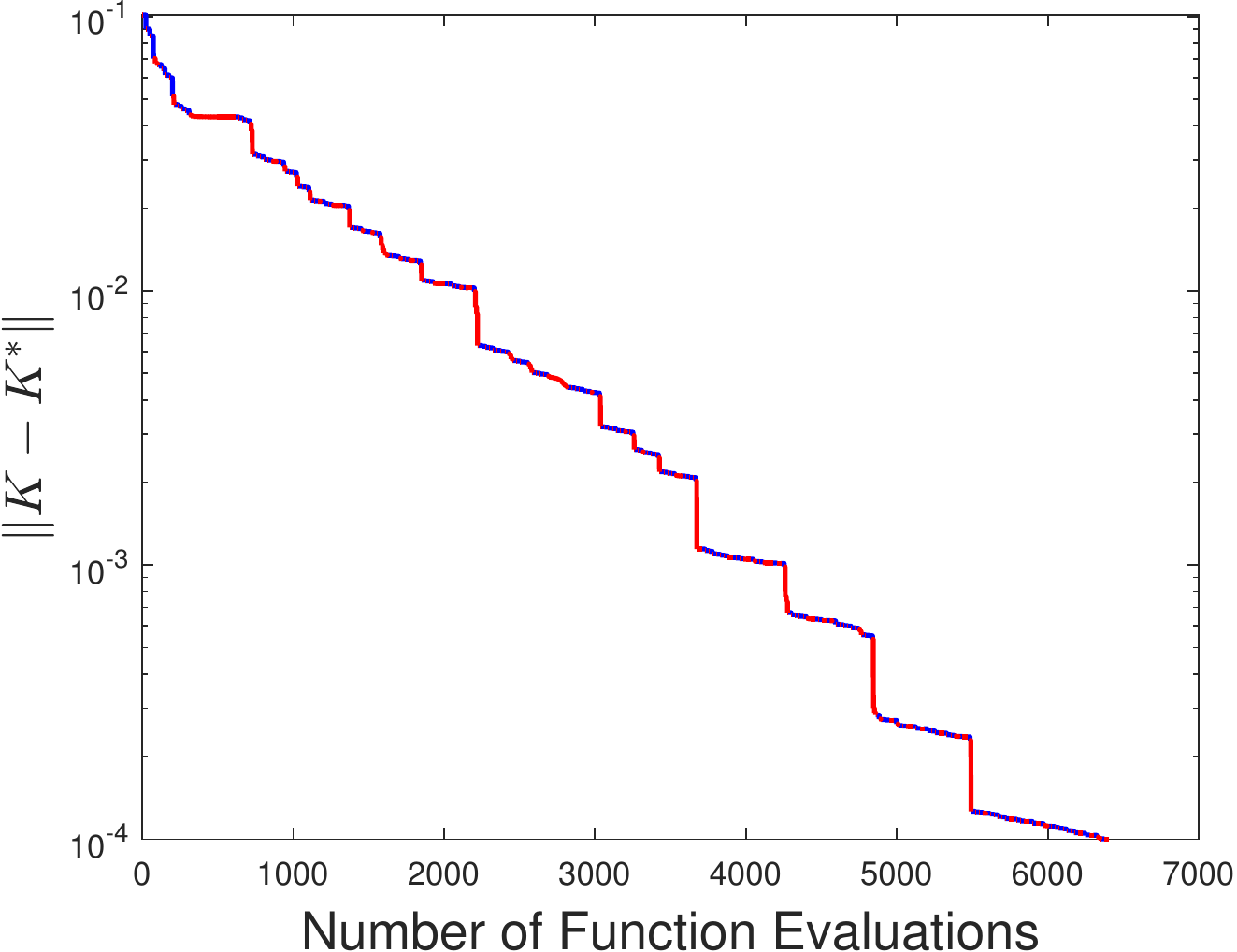}
		\caption{}
		\label{fig:gc_lqr_too_small_eta_min}
	\end{subfigure}	
	\hfill
	\begin{subfigure}[b]{0.3\textwidth}
		\centering
		\includegraphics[width=\textwidth,trim = 0 0 35 35]{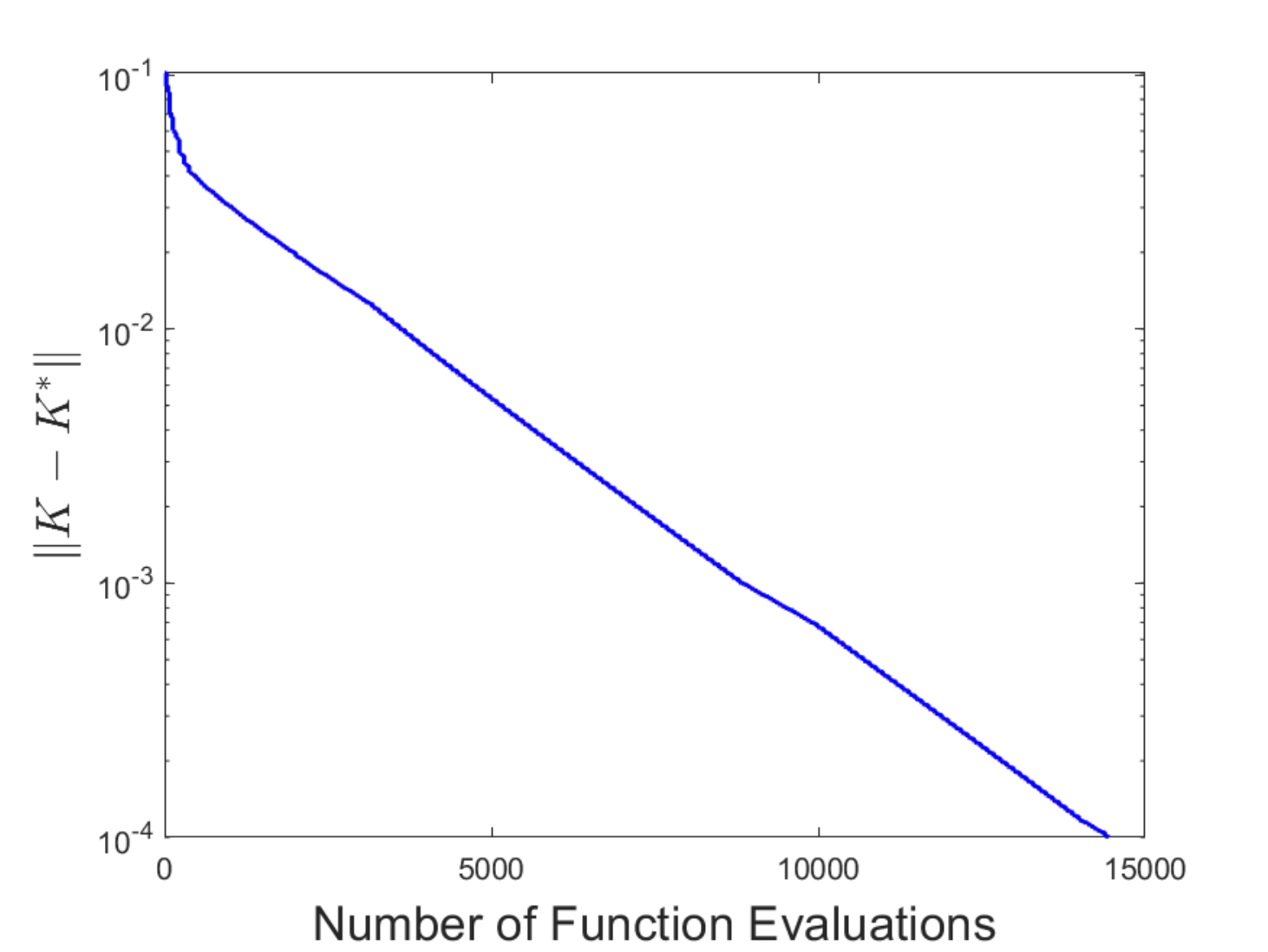}
		\caption{}
		\label{fig:gc_lqr_large_eta_min}
	\end{subfigure}	
	\hfill%
}
	\caption{Effect of the minimum step size $\etamin$ on the convergence of the GC algorithm: (a) $ \etamin = 5 \times 10^{-4} $. (b) $\etamin = 5 \times 10^{-9}$. (c) $ \etamin = 0.5 $. Choosing $\etamin$ either too small or too large increased the total number of function evaluations. (Blue: model-free regime. Red: model-based regime.)}
	\label{fig:lqr_eta_min}
\end{figure}

Since backtracking line search was what prevented the GC algorithm from staying in the model-based regime, we changed the minimum step size $\etamin$ from $ 0.05 $ to $5 \times 10^{-4} $. The result is shown in Fig.~\ref{fig:gc_lqr_small_eta_min}, from which it can be seen that the algorithm stayed
longer in the model-based regime, and the total number of function evaluations was smaller compared with Fig.~\ref{fig:gc_lqr}. However, upon further decreasing $\etamin$ to $ 5 \times 10^{-9}$, the total number of function evaluations increased, as shown in Fig.~\ref{fig:gc_lqr_too_small_eta_min}.  Recall that each attempt within backtracking line search decreases the previous step size by a constant factor $\beta$ and uses one function evaluation to determine whether the new step size is appropriate. For a smaller $\etamin$, in order to declare the failure of line search, backtracking line search requires more attempts and hence more function evaluations before the proposed step size $\eta$ reaches $ \etamin $. To reduce the total number of iterations for reaching a desired level of suboptimality, one should choose $\etamin$ to balance between the number of iterations and the number of function evaluations (per iteration) in the model-based regime. Choosing a smaller $\etamin$ will make the algorithm stay in the model-based regime for more iterations and hence better exploit model information at the expense of more function evaluations per iteration.

We also experimented with a larger value of $\etamin$ by choosing $\etamin = 0.5$, the result of which is shown in Fig.~\ref{fig:gc_lqr_large_eta_min}. As expected, a larger $ \etamin $ caused the algorithm to spend fewer iterations in the model-based regime because the line search termination condition was tightened. In fact, for $ \etamin = 0.5 $, the GC algorithm did not enter the model-based regime at all; the algorithm only stayed in the model-free regime and behaved identically to the model-free method in~\cite{qu_combining_2020} shown in Fig.~\ref{fig:qu_lqr}.

\section{Conclusion}

We have proposed a framework that combines model-based and model-free methods
for optimal control. The framework formulates the optimal control problem as a composite
optimization problem whose objective function is described by a sum of two terms: a
model-based term whose explicit expression can be obtained from an approximate model of the system dynamics and a residual term that captures the unknown modeling error.
To solve the composite optimization problem, we have developed a hybrid gradient-based
algorithm that uses gradient information from the approximate model with
compensation from intermittent model-free updates. Both theoretical
analysis and numerical experiments show that the algorithm uses fewer 
function evaluations than a purely model-free policy gradient algorithm for reaching the same level of suboptimality.

\bibliographystyle{plain} 
\bibliography{GC_L4DC,ref-sh} 

\end{document}